\documentclass[twoside, 10pt]{article}
 \RequirePackage{amsgen}
\RequirePackage{amsmath} \RequirePackage{amstext}
\RequirePackage{amsbsy} \RequirePackage{amsopn}
\RequirePackage{amsthm} \RequirePackage{amssymb}
\RequirePackage{epsfig}

\usepackage[mathscr]{eucal}

\theoremstyle{plain}
\newtheorem{theorem}{Theorem}[section]
\newtheorem{proposition}[theorem]{Proposition}
\newtheorem{lemma}[theorem]{Lemma}

\newtheorem{corollary}[theorem]{Corollary}

\let\a\alpha

\let\c\chi

\let\e\epsilon

\let\e\varepsilon

\def\F{{\mathbb F}}

\def\Z{{\mathbb Z}}
\def\Zp{{\Z_{p}}}
\def\Q{{\mathbb Q}}

\begin{document}

\title{{\bf Toward the ergodicity of  $p$-adic 1-Lipschitz functions represented by the van der Put series}}
\author{{ Sangtae Jeong}
\\Department of Mathematics, Inha University, Incheon, Korea
402-751}

\footnotetext{

\noindent {\rm Email}: stj@inha.ac.kr}

\bibliographystyle{alpha}
\maketitle

\begin{abstract}
Yurova \cite{Yu}  and Anashin et al. \cite{AKY1, AKY2} characterize the ergodicity of a 1-Lipschitz function on $\Z_2$ in terms of the van der Put expansion.
Motivated by their recent work, we provide the sufficient conditions for the ergodicity of such a function defined on a more general setting $\Z_p$.
 In addition, we provide alternative proofs of two criteria (because of \cite{AKY1, AKY2} and \cite{Yu} ) for an ergodic 1-Lipschitz function on $\Z_2,$ represented by both the Mahler basis and the van der Put basis.

\end{abstract}

\section{\bf Introduction}

The ergodic theory of $p$-adic  dynamical systems is an important part of non-Archimedean dynamics, and represents  a rapidly developing discipline that has
recently demonstrated its effectiveness in various areas such as computer science, cryptology, and numerical analysis, among others.
For example, as shown in \cite{KS}, it is useful to have $2
$-adic ergodic functions in constructing long-period
pseudo-random sequences in stream ciphers. For more details on such applications, we refer the reader to \cite{AK}  and the references therein.

As a substitute for the Mahler basis, the van der Put basis has recently been employed as a useful tool for  building on the
ergodic theory of $p$-adic dynamical systems. Indeed, Yurova \cite{Yu}  and Anashin et al. \cite{AKY1, AKY2} provide the
criterion for the ergodicity of 2-adic 1-Lipschitz functions, in terms of the van der Put expansion.
Their proof of this criterion relies on Anashin's criterion for  1-Lipschitz functions on $\Z_2$ in terms of the Mahler expansion.
Given the characteristic functions of $p$-adic balls, it is analyzed in  \cite{AKY2} that the van der Put basis
has more advantages than the Mahler basis in evaluating representations and that it is more applicable to $T$-functions or 1-Lipschitz functions.

On the other hand, on the function field side of non-Archimedean dynamics,  Lin et al. \cite{LSY} present an ergodic theory parallel to \cite{An1} and \cite{AKY1,AKY2} by
using both Carlitz-Wagner basis and an analog of the van der Put basis.
Along this line, Jeong \cite{J} uses the digit derivative basis to develop a corresponding theory parallel to \cite{LSY}.

The purpose of the paper is to provide the sufficient conditions under which 1-Lipschitz functions on $\Z_p$ represented by the van der Put series are ergodic.
In addition, we provide  alternative proofs of two known criteria for an ergodic 1-Lipschitz function on $\Z_2$ in  terms of both the Mahler basis and the van der Put basis.
We also present several equivalent conditions that may be needed to provide a complete description of the ergodicity of   1-Lipschitz functions defined on a more general setting
$\Z_p.$ The main idea behind this paper comes from Lin et al's work \cite{LSY} on $\F_2[[T]],$ and  Anashin et al.'s work \cite{AKY2} on $\Z_2.$

The rest of this paper is organized as follows: Section 2 recalls some prerequisites in non-Archimedean dynamics, including two known results for the ergodicity of 1-Lipschitz functions on $\Z_2$
in terms of the Mahler basis and the van der Put basis. Section 3 presents the main results and alternate proofs of two criteria  for
 an ergodic 1-Lipschitz function on $\Z_2.$  Section 4 employs our results or Anashin's to re-prove then ergodicity of a polynomial over $\Z_2$ in terms of its coefficients.


\section{Ergodic theory of $p$-adic integers}
We recall the existing results for the measure-preservation and ergodicity of 1-Lipschitz functions $f: \Z_2 \rightarrow \Z_2$ in terms of both the
 Mahler expansion and the van der Put expansion.

\subsection{Preliminaries for $p$-adic dynamics}
We recall  the elements of $p$-adic dynamical systems on $\Z_p.$
Let $p$ be a prime and $\Zp$ be the ring of  $p$-adic integers with the quotient field $\Q_p.$
Let $|?|=|?|_p$  be the (normalized) absolute value on $\Q_p$ associated with the additive valuation ${\rm ord}$ such that
$ |x|_{p} = p^{-{\rm ord}(x)}$ for $x \not =0$ and $|0|=0$ by convention.

The space $\Zp $ is equipped with the natural probability measure $\mu_p,$ which is normalized so that $\mu_p(\Z_p)=1.$
Elementary $\mu_p$-measurable sets are $p$-adic balls by which we mean a set $a+ p^k\Z_p$ of radius $p^{-k}$ for $a \in \Z_p.$ We define the volume of this ball as $\mu_p(a+ p^k\Z_p)=1/p^k.$

A $p$- adic dynamical system on $\Z_p$ is understood as a triple $(\Z_p, \mu_p, f),$ where $f:\Z_p \rightarrow \Z_p$ is a measurable function.
Starting with any chosen point $x_0$ (an initial point), the trajectory of $f$ is a sequence of elements of the form
$$ x_0, x_1=f(x_0), \cdots, x_{i}=f(x_{i-1})=f^i(x_0)\cdots.$$

Here we say that $f$ is bijective modulo $p^n$ for a positive integer $n$ if a sequence of $p^n$ elements
$ x_0, x_1=f(x_0), \cdots, f^{p^n-1}(x_0)$ is distinct in the factor ring $\Zp/p^n\Z_p.$ And $f$ is said to be transitive modulo $p^n$ if the above sequence forms a single cycle in $\Zp/p^n\Z_p.$
We say that a function $f : {\Z_p} \rightarrow {\Z_p}$ of the measurable space ${\Z_p}$
 with the Haar measure $\mu=\mu_p$ is {\em measure-preserving}  if
 $\mu(f^{-1}(S))=\mu(S)$ for each measurable subset $S \subset {\Z_p}.$
 A measure-preserving function $f: \Z_p \rightarrow {\Z_p}$ is said to be {\em ergodic} if it has no proper
 invariant subsets. That is,
 if $f^{-1}(S) =S$ for a measurable subset, then $S \subset {\Z_p}$ implies that
 $\mu(S)=1$ or $\mu(S)=0.$
We say that $f: \Z_p\rightarrow \Z_p $ is 1-Lipschitz (or compatible) if  for all $x,y \in  \Z_p,$
$$|f(x)-f(y)|_{p} \leq |x-y|_{p}.$$  Note that a 1-Lipschitz function $f$ is continuous on $\Z_p.$
We observe that the 1-Lipschitzness condition has several equivalent statements:

(i) $|f(x+y)-f(x)|_{p} \leq |y|_{p}$ for all $x,y \in \Z_p;$

(ii) $|\frac{1}{y}(f(x+y)-f(x))|_{p} \leq 1$ for all $x \in \Z_p$ and all $ y \not =0 \in  \Z_p;$

(iii) $ f(x+p^{n}\Z_p) \subset f(x)+ p^{n}\Z_p$ for all $x\in  \Z_p$ and any integer $n\geq1;$

(iv) $f(x) \equiv f(y) \pmod{p^n}$ whenever $x \equiv y \pmod{p^n}$ for any integer $n\geq1.$

For later use, We recall the following criteria for the measure-preservation and ergodicity of a 1-Lipschitz function:
\begin{proposition}{\bf \cite{An1,DP}}\label{mper}
 Let  $f: \Z_p\rightarrow \Z_p $ be a 1-Lipschitz function.

\noindent(i) The following are equivalent:

(1) $f$ is measure-preserving;

(2) $f$ is bijective modulo $p^n$ for all integers $n>0;$

(3) $f$ is an isometry, i.e., $|f(x)-f(y)|_{p} = |x-y|_{p}$ for all $x,y \in \Z_p.$

\noindent(ii) $f$ is ergodic if and only if it is transitive modulo $p^n$ for all integers $n>0$.
\end{proposition}

Throughout this paper, we denote the greatest integer that is less than or equal to a real number $a$ by $\lfloor a \rfloor.$

\subsection{Mahler  basis and ergodic functions on $\Z_2$}

It is well known \cite{Ma1,Ma2} that every continuous function $f:\Z_p \rightarrow \Z_p$  is represented by the Mahler interpolation series

\begin{eqnarray}\label{Maexp}
f(x)=\sum_{m= 0}^{\infty}a_m \binom{x}{m},
\end{eqnarray}
where $a_m\in \Zp$ for $m=0, \cdots $ and
the binomial coefficient functions are define by $$\binom{x}{m}=\frac{1}{m!}x(x-1) \cdots (x-m+1)~~(m\geq1),~~\binom{x}{0}=1.$$
We now state Anashin's characterization results for the measure-preservation and ergodicity of 1-Lipschitz functions in terms of the
coefficients of the Mahler expansion.
\begin{theorem}{\bf Anashin \cite{An1, AK}}\label{Ana-binom}

(i) The function $f$ in Eq.(\ref{Maexp}) is 1-Lipschitz on $\Z_p$  if and only if the following conditions are satisfied: For all $ m \geq 0,$
\begin{eqnarray*}\label{Mlipcon}
|a_m | \leq |p|^{\lfloor {\rm log}_{p} m \rfloor}.
\end{eqnarray*}

(ii)The function $f$ is a measure-preserving 1-Lipschitz function on $\Z_p$  whenever the following conditions are satisfied:
\begin{eqnarray*}
 |a_1| &=&1;\\
 |a_m | &\leq& |p|^{\lfloor {\rm log}_{p} m \rfloor +1}~ {\rm for~ all~} m \geq 2.
\end{eqnarray*}

(iii) The function $f$ is an ergodic 1-Lipschitz  function on $\Zp$ whenever the following conditions are satisfied:
\begin{eqnarray*}
 a_0  &\not \equiv& 0 \pmod{p};\\
  a_1  &\equiv& 1 \pmod{p}; \\
 a_m &\equiv& 0 \pmod {p^{\lfloor {\rm log}_{p} (m+1) \rfloor +1}}  ~{\rm  for ~all}~ m\geq 2.
\end{eqnarray*}

(iv)  The function $f$ is an ergodic 1-Lipschitz  function on $\Z_2$ if and only if the following conditions are satisfied:
\begin{eqnarray*}
 a_0  & \equiv& 1 \pmod{2};\\
a_1 & \equiv& 1  \pmod{4};\\
 a_m &\equiv& 0 \pmod {2^{\lfloor {\rm log}_{2} (m+1) \rfloor +1}}  ~{\rm  for ~all}~ m\geq 2.
\end{eqnarray*}

\end{theorem}

Anashin's proof of Theorem \ref{Ana-binom} (iv) relies on a criteria, namely Theorem 4. 39 in  \cite{AK}, based on the algebraic normal form of Boolean functions which determines the  measure-preservation and ergodicity of 1-Lipschitz
functions. The tricky part of his proof is to use this criterion to derive a recursive formula for the coefficients of Boolean coordinates
of a 1-Lipschitz function $f.$  As an easy corollary of this theorem, Anashin \cite{AK} derives the following
result, which turns out to be a useful method for constructing measure-preserving (ergodic) 1-Lipschitz
functions out of an arbitrary 1-Lipschitz function. Here recall that $\Delta$ is the difference operator defined by $ \Delta f(x) =f(x+1)-f(x).$
\begin{corollary}\label{1+xDE}
Every ergodic (resp. every measure-preserving)  1-Lipschitz function $f:\Z_2 \rightarrow \Z_2$ can
be represented as $f(x)= 1+x +2\Delta g(x)$
(resp. as $f(x)=d +x +2 g(x)$) for a suitable constant $d \in \Z_2$ and a suitable1-Lipschitz function $g:\Z_2 \rightarrow \Z_2 $
and vice versa, and every function $f$ of the above form is an ergodic (thus, measure-preserving) 1-Lipschitz function.
\end{corollary}

In this paper, using the van der Put basis, we re-prove  this corollary and use it to
provide an alternative proof of Theorem \ref{Ana-binom} (iv).

For later use, we recall Lemma 4. 41 in \cite{AK}, from which we deduce one of the main results: Theorem 3. 8.
 \begin{lemma}\label{dxperg}
 Given a 1-Lipschitz function $g:\Z_p \rightarrow \Z_p$ and a $p$-adic integer $d \not \equiv 0 \pmod{p}$, the
 function $f(x)= d + x +p\Delta g(x)$ is ergodic.
 \end{lemma}

\subsection{Van der Put basis and ergodic functions on $\Z_2$ }

We introduce a sequence of the van der Put basis $\chi(m,x)$ on the ring $\Zp$ of $p$-adic integers.
For an integer $m>0$ and $x \in \Zp,$ we
define
\[ \chi(m,x)   =
                   \left \{ \begin{array}{ll}
1       & if ~|x-m| \leq p^{ -\lfloor {\rm log}_{p}(m)\rfloor  -1};\\
0   & ~otherwise
 \end{array}
 \right. \]
and
\[ \chi(0,x)   =
                   \left \{ \begin{array}{ll}
1       & ~if |x| \leq p^{-1};\\
0   & ~otherwise.
 \end{array}
 \right. \]

Indeed, the van der Put basis is a characteristic function of the balls $B_{p^{ -\lfloor {\rm log}_{p}(m)\rfloor  -1}}(m)~(m\geq 1)$ and $B_{1/p}(0).$
By the well-known result of van der Put \cite{vPut}(see also \cite{Ma2,S}), we know that
every continuous function $f:\Z_p \rightarrow \Z_p$  is represented by the van der Put series:
\begin{eqnarray}\label{putseries}
f(x)
=\sum_{m=0}^{\infty}B_{m}\chi(m,x),
\end{eqnarray}
where $B_m\in \Zp$ for $m=0, \cdots.$
We write an integer $m>0$ in the $p$-adic form as
$$ m =m_0+m_1p+ \cdots + m_sp^s (m_s \not =0)$$
From the $p$-adic representation of $m,$ we see that
$$ s = \lfloor {\rm log}_{p}(m)\rfloor = ({ \rm  the~number~of~ digits~ in ~the~ p-adic~ form~ of~ m} ) - 1$$
by assuming that $ \lfloor {\rm log}_{p} 0\rfloor =0.$
Throughout this paper, we set $$q(m)=m_sp^s,~~ m\underline{~~} =m-q(m).$$
Then we have $m=m\underline{~~} +q(m).$
What is important here is that
the expansion coefficients $\{B_{m} \}_{m\geq0}$ can be recovered by the following formula:
\begin{eqnarray}\label{Bmform}
B_m   =
                   \left \{ \begin{array}{ll}
f(m)-f(m-q(m)) =f(m)-f(m\underline{~~})       & ~if ~m\geq p;\\
f(m)   & ~otherwise.
 \end{array}
 \right.
\end{eqnarray}

As a result parallel to Theorem \ref{Ana-binom}, we state the following characterization for the ergodicity of  a  1- Lipschitz function $f$ in terms of the van der Put expansion.
Indeed,  Yurova \cite{Yu} and Anashin et al \cite{AKY1,AKY2} deduce Theorem \ref{putcri2} from Corollary \ref{1+xDE}.
However, in Section 3.4 we provide an alternate proof of it independently of Theorem \ref{Ana-binom}.

\begin{theorem}{\bf Yurova \cite{Yu} and Anashin et al \cite{AKY1,AKY2}}\label{putcri2}

(i) The function $f$ in Eq.(\ref{putseries}) is 1-Lipschitz on $\Z_p$  if and only if
the following conditions are satisfied: For all $ m \geq 0,$
$$ |B_m | \leq |p|^{\lfloor {\rm log}_{p} m \rfloor}.$$

(ii) The 1-Lipschitz function $f$ on $\Z_2$ represented by the van der Put series
\begin{eqnarray}\label{putseries2}
f(x)=b_0\chi(0,x) + \sum_{m=1}^{\infty}2^{\lfloor {\rm log}_{2}m\rfloor }b_m\chi(m,x)~~(b_m \in \Z_2)
\end{eqnarray}
is measure-preserving  on $\Z_2$ if and only if

 (1) $b_0 + b_1 \equiv 1 \pmod{2};$

 (2) $|b_m| =1   ~{\rm  for ~all}~ m\geq 2.$

(iii) The 1- Lipschitz function $f$ represented by the van der Put series
 in Eq.(\ref{putseries2})
is ergodic on $\Z_2$  if and only if the following conditions are satisfied:

(1) $b_0 \equiv 1 \pmod{2};$

(2) $ b_0 +b_1 \equiv 3 \pmod{4};$

(3) $b_2 +b_3 \equiv 2 \pmod{4};$

(4) $|b_m| =1$ for all $m\geq2;$

(5) $\sum_{m=2^{n-1}}^{2^n -1}b_m \equiv 0 \pmod{4}$ for all $n\geq 3.$

\end{theorem}

\section{ Ergodic  $p$-adic maps on $\Z_p$}
In this section, which is divided into four subsections, we present the main results of this paper. We first re-prove the  1-Lipschitz property of $p$-adic functions  represented by the van der Put series and then
provide the sufficient conditions for the measure-preservation of such functions. Using the latter conditions and Corollary \ref{dxperg}, we
provide several conditions for coefficients under which 1-Lipschitz functions on $\Z_p$ are ergodic. In addition, we present several equivalent conditions for the van Put coefficients for
$p$-adic functions. We use these equivalent conditions for $p=2$  to provide an alternate proof of
Anashin et al.'s criterion in \cite{AKY1,AKY2}, that is, Theorem \ref{putcri2} (iii).
Finally, using this fact, we provide a simple proof of  Anashin's criterion in \cite{An1}, that is, Theorem \ref{Ana-binom} (iv).

\subsection{Measure-preserving 1-Lipschitz functions on $\Z_p$}

We provide the  necessary and sufficient conditions for $f$ to be 1-Lipschitz in terms of the coefficients of the van der Put expansion. This result is known \cite{AKY2}, but we provide a simple proof.

\begin{proposition}\label{lips}
Let $f(x)=\sum_{n=0}^{\infty}B_{m}\chi(m,x):\Z_p \rightarrow \Z_p$  be a continuous function  represented
by the van der Put series.
Then $f$ is 1-Lipschitz if and only if  $|B_m| \leq p^{- \lfloor {\rm log}_{p}m\rfloor}$ for all nonnegative integers $m.$
\end{proposition}
\begin{proof}
Assuming that $f$ is 1-Lipschitz, by the formula for $B_m$ in Eq.(\ref{Bmform}) we compute the following for $m\geq p:$
$$|B_m| =|f(m)-f(m-q(m))| \leq |q(m)| = p^{- \lfloor {\rm log}_{p}m\rfloor}.$$
Then the result follows by noting that the inequality holds trivially for $0 \leq m <p.$

Conversely, assuming that the inequality holds, we first observe that
if $ x\equiv y \pmod{p^n},$ then $\chi(m,x) =\chi(m,y)$ for all $0 \leq m <p^n.$
Then, under the assumption that $ x\equiv y \pmod{p^n},$
we compute $$f(x)-f(y) =\sum_{m=0}^{\infty}B_m(\chi(m,x)-\chi(m,y))\equiv \sum_{m=0}^{p^n-1}B_m(\chi(m,x)-\chi(m,y))
\equiv 0 \pmod{p^n},$$where the last congruence follows from the observation.
Therefore, the result follows.
\end{proof}

We now provide the sufficient conditions for a  1- Lipschitz function $f$ on $\Z_p$ to be measure-preserving.

\begin{theorem}\label{MPinp}
The 1- Lipschitz function $f(x)
=\sum_{n=0}^{\infty}B_{m}\chi(m,x):\Z_p \rightarrow \Z_p$
is measure-preserving whenever the following conditions are satisfied:

(1) $\{B_0, B_1, \cdots, B_{p-1}\}$ is distinct modulo $p;$

(2) $B_m \equiv q(m) \pmod{p^{\lfloor {\rm log}_{p}m\rfloor  +1}}$ for all $m\geq p.$
\end{theorem}

\begin{proof}

By Proposition \ref{mper}, it suffices to show  that $f$ is bijective modulo $p^n$ for every positive integer $n.$ Because  $\Z_p/p^n\Z_p$ is a finite set,
it is also equivalent to showing that $f$ is injective modulo $p^n.$
 Suppose that $f$ is not injective modulo $p^n$ for some integer $n>0.$ Then we observe
 $n\geq 2$ because $f$ is bijective modulo $p$, by assumption (1).
Here we see that there exist $a$ and $b$ in $\Z_p/p^n\Z_p$ with
 $a \not \equiv b \pmod{p^n}$ such that $f(a)\equiv f(b) \pmod{p^{n}}.$
Write
\begin{eqnarray*}
a &=& a_0 +a_1p+ \cdots +a_{n-1}p^{n-1}~~{\rm with}~~0 \leq a_i <p,  \\
b &=& b_0 +b_1p+ \cdots +b_{n-1}p^{n-1}~~{\rm with}~~  0 \leq b_i <p .
 \end{eqnarray*}

Since $a \not \equiv b \pmod{p^n}, $ these exists a nonnegative integer $r$ such that $a_r \not =b_r,$ for which we may
assume that $r$ is the minimal index (thus $r\leq n-1$.).
Set
\begin{eqnarray*}
m_1 &=& a_0 +a_1p+ \cdots +a_{r}p^{r}, \\
m_2 &=& b_0 +b_1p+ \cdots +b_{r}p^{r}.
 \end{eqnarray*}
We can assume that $a_r \not =0$ and $b_r \not =0.$ Otherwise, the following argument can be applied in a similar fashion.
Because $f$ is 1-Lipschitz, we first deduce the following inequality:
\begin{eqnarray*}
|f(m_1)-f(m_2)| &=&|f(m_1)-f(a)+f(a)-f(b) +f(b)-f(m_2)|\\
&\leq& {\rm max} \{ |f(m_1)-f(a)|,|f(a)-f(b)|, |f(b)-f(m_2)|  \}\\
&\leq& |p|^{r+1}.
 \end{eqnarray*}
Then we have $B_{m_1} =f(m_1)-f(m_1\underline{~~})$ and $B_{m_2} =f(m_2)-f(m_2\underline{~~}).$
Since $m_1\underline{~~} =m_2\underline{~~},$ the preceding inequality yields
$$B_{m_1}-B_{m_2} =f(m_1)-f(m_2) \equiv 0 \pmod{p^{r+1}}.$$
On the other hand, by assumption (2), we have
$$ B_{m_1}-B_{m_2} \equiv q(m_1)-q(m_2)=(a_r-b_r)p^r \pmod{p^{r+1}}.$$ Because $a_r \not =b_r,$ the preceding congruence
gives $ B_{m_1} -B_{m_2} \not \equiv 0 \pmod{p^{r+1}}.$ Therefore, we have a contradiction.

\end{proof}
We note that condition (1) in Theorem \ref{MPinp} is well known to be equivalent to the following congruence(see Lemma 7.3. in \cite{LN}): For any prime $p>2,$
$$ \sum_{m=0}^{p-1}{B_m}^k \equiv
                \left \{ \begin{array}{ll}
0    \pmod{p}   & ~if ~ 0\leq k \leq p-2;\\
-1   \pmod{p}   & ~if ~ k =p-1.
 \end{array}
 \right.
 $$
For the converse of Theorem \ref{MPinp}, we have the following
\begin{proposition}\label{MPinp2}
 Let $f(x)
=\sum_{n=0}^{\infty}B_{m}\chi(m,x):\Z_p \rightarrow \Z_p$
be a measure-preserving 1- Lipschitz function.
Then we have the following:

(1) $\{B_0, B_1, \cdots, B_{p-1}\}$ is distinct modulo $p$.

(2) $|B_m| =|q(m)|=|p|^{\lfloor {\rm log}_{p}m\rfloor }$ for all $m\geq p$

\end{proposition}

\begin{proof}
It is easy to see that part (1) follows from Proposition \ref{mper}.
To deduce part (2), write $m\geq p$ as $m=m\underline{~~} +q(m).$
Becasue $f$ is a measure-preserving 1-Lipschitz function, by Proposition \ref{mper}(3) and  Eq.(\ref{Bmform}),
we have $$|B_m|=|f(m)-f(m\underline{~~})|=|m-m\underline{~~}|=|q(m)|,$$
which completes the proof.
\end{proof}

From Proposition \ref{MPinp2}, we see that the conditions in Theorem  \ref{MPinp} are necessary for the case in which $p=2,$ and therefore we provide an alternate proof
of Theorem \ref{putcri2} (ii).

\begin{proposition}\label{Mp-sumform}
 Let $f(x)
=\sum_{n=0}^{\infty}B_{m}\chi(m,x):\Z_p \rightarrow \Z_p$
be a measure-preserving 1- Lipschitz function. For $p^{n-1} \leq m \leq p^n -1~(n\geq2),$
set $$B_m=p^{n-1}b_m =p^{n-1}(b_{m0} +b_{m1}p +\cdots )~~(b_{m0} \not =0, 0 \leq b_{mi} \leq p-1, i=0,1 \cdots). $$
Then, for all $n\geq2,$ we have
$$\sum_{m=p^{n-1}}^{p^n-1}B_m \equiv \frac{1}{2}(p-1)p^{2n-1} +T_np^n \pmod{p^{n+1}},$$
where $T_n$ is defined by $T_n =\sum_{m=p^{n-1}}^{p^n-1}b_{m1}.$
\end{proposition}
\begin{proof}
For given $m,$ write $m=ip^{n-1} +j$ with $ 1 \leq i \leq p-1,$ $0 \leq j\leq p^{n-1}-1$  and $n\geq 2.$
We show  that for any fixed $j,$ $\{ b_{ip^{n-1} +j,0}\}_{1 \leq i \leq p-1}$
is distinct, that is, a permutation of $1, \cdots, p-1.$
For such $j,$ we consider $B_{ip^{n-1} +j}$ for all $i=1, \cdots,p-1.$
Because $f$ is a measure-preserving 1- Lipschitz function, by Eq. (\ref{Bmform}) and Proposition \ref{mper} (3), we have the following for $ 1 \leq i,i' \leq p-1:$
$$B_{ip^{n-1} +j}-B_{i'p^{n-1} +j}=f(ip^{n-1} +j)-f(i'p^{n-1} +j)\equiv (i-i')p^{n-1} \pmod{p^n}.$$
From the definition of $B_m$ in the statement, we also have $$ B_{ip^{n-1} +j}-B_{i'p^{n-1} +j} \equiv (b_{ip^{n-1} +j,0}-b_{i'p^{n-1} +j,0})p^{n-1} \pmod{p^n}.$$
By equating these two congruence relations, we see that $i \not =i'$ if and only if $b_{ip^{n-1} +j,0} \not =b_{i'p^{n-1} +j,0},$ which implies the assertion.
Here, by using the assertion to compute the congruence
$$\sum_{m=p^{n-1}}^{p^n-1}B_m \equiv p^{n-1}\sum_{j=0}^{p^{n-1}-1}\sum_{i=1}^{p-1}b_{ip^{n-1} +j,0}+T_np^n \pmod{p^{n+1}},$$
we obtain the  desired result.
\end{proof}

\subsection{ Some conditions for ergodic functions on {$\Z_p$}}
In this subsection, we provide several conditions for $B_m$ under which a measure-preserving 1-Lipschitz function $f$ on $\Z_p$
is ergodic. Therefor, Anashin et al.'s result \cite{AKY1, AKY2} can be extended to a general case for a prime $p.$

To begin with, we have the connection between the van der Put expansions of  a continuous function $f$ and $\Delta f.$

\begin{proposition}\label{lipcond}
If a 1-Lipschitz (continuous) function $f=\sum_{m=0}^{\infty}{B_m}\c(m,x):\Z_p \rightarrow \Z_p$
is of the form
$f(x)= \Delta g(x)$ for some 1-Lipschitz function $g=\sum_{m=0}^{\infty}{\tilde B}_m\c(m,x),$ then
we have
\begin{eqnarray}
B_m &=& {\tilde B}_{m +1}-  {\tilde B}_{m}  ~\quad \quad \quad \quad \quad if ~0 \leq  m\leq p-2;\label{ls1}\\
 &=& {\tilde B}_{p} +{\tilde B}_{0}- {\tilde B}_{p-1}   ~\quad \quad \quad if~ m=p-1;\label{ls2}\\
 &=& {\tilde B}_{m+1}-{\tilde B}_{m}  ~\quad \quad \quad  if ~ m \not = p^{n-1}-1+m_{n-1}p^{n-1}, p^{n-1} \leq m \leq p^n-1,n\geq 2;\label{ls3}\\
 &=& {\tilde B}_{m+1}-{\tilde B}_{m}-{\tilde B}_{p^{n-1}} ~if~m = p^{n-1}-1+m_{n-1}p^{n-1}, 1 \leq m_{n-1} \leq p-1, n \geq 2. \label{ls4}
\end{eqnarray}
\end{proposition}

\begin{proof}
For given $g(x)=\sum_{m=0}^{\infty}{\tilde B}_m\chi(m,x),$
write $g(x+1)=\sum_{m=0}^{\infty}{\bar B}_m\chi(m,x)$ in terms of the van der Put expansion.
We first need to determine the relationship between ${\bar B}_{m}'s$ and ${\tilde B}_m's.$
By Eq.(\ref{Bmform}), it is easy to see that for $ 0 \leq m < p-1,$
${\bar B}_{m} =g(m+1)= {\tilde B}_{m+1}$
and that  ${\bar B}_{p-1} =g(p)= {\tilde B}_{p} +{\tilde B}_{0}.$
Write $m$ in the $p$-adic form as
$m=m_0 +m_1p+ \cdots +m_{n-1}p^{n-1}$ with $ 0 \leq m_i <p,$ $m_{n-1} \not =0,$ and $n \geq 2.$
If $m\not = p^{n-1}-1 + m_{n-1}p^{n-1},$ then we have
$q(m+1)=q(m),$ and therefore, by Eq.(\ref{Bmform}), we again
have
$${\bar B}_m =g(m+1)-g(m+1 -q(m))= g(m+1)-g(m+1-q(m+1)) ={\tilde B}_{m+1}.$$
If $m =p^{n-1}-1 + m_{n-1}p^{n-1}\leq p^n-1$ with $  1 \leq m_{n-1} \leq p-1,$
then $q(m+1)=q(m)+p^{n-1},$ and therefore
we have
\begin{eqnarray*}
{\bar B}_m &=&g(m+1)-g(m+1 -q(m)) =g((m_{n-1}+1)p^{n-1})-g(p^{n-1})\\
&=& g((m_{n-1}+1)p^{n-1}) -g(0) -(g(p^{n-1}) -g(0)) \\
&=& {\tilde B}_{m+1} -{\tilde B}_{p^{n-1}}.
\end{eqnarray*}
The result follows by equating the coefficients of $f(x)$ and $\Delta g(x).$
\end{proof}

A natural question arising from Proposition \ref{lipcond} is under what conditions for coefficients of a 1-Lipschitz function $f$ we have $f$ of the form $f(x)=\Delta g(x)$ for a suitable 1-Lipschitz function $g.$
The following result answers this question:

\begin{proposition}
Let $f=\sum_{m=0}^{\infty}{B_m}\c(m,x):\Z_p \rightarrow \Z_p$ be a 1-Lipschitz function satisfying

(1) $\sum_{m=0}^{p-1}{B_m} \equiv 0 \pmod{p};$

(2) $\sum_{m=p^{n-1}}^{p^n -1} B_m \equiv 0 \pmod{p^{n}}$ for all $n\geq 2.$

Then there exists a 1-Lipschitz function $g(x)$ such that $f(x)=\Delta g(x).$
\end{proposition}
\begin{proof}
By Proposition \ref{lipcond}, we need to find a 1-Lipschitz function $g(x)=\sum_{m=p}^{\infty}{\tilde B}_{m}\c(m,x)$ whose coefficients ${\tilde B}_{m}$
satisfy a system of linear equations in Eqs.(\ref{ls1})-(\ref{ls4}). We view ${\tilde B}_{m}$ as the variables required for solving  a system of linear equations for
countably many variables ${\tilde B}_{m}$.  As in \cite{AKY2} for the case $p=2,$ we inductively construct a sequence of $p$-adic integers $\{{\tilde B}_{m}\}_{m\geq 0}$
with ${\tilde B}_{m} \equiv 0 \pmod{p^{\lfloor {\rm log}_{p}m\rfloor}}$ satisfying the above linear system.
From a system of linear equations in Eqs.(\ref{ls1}) and (\ref{ls2}),
we find $p$-adic integers ${\tilde B}_{0}, \cdots {\tilde B}_{p} \in \Z_p$ such that
 \begin{eqnarray*}
 {\tilde B}_{m}&=&{\tilde B}_{0} +\sum_{i=0}^{m-1}B_i~~~(m=1,\cdots p-1); \\
 {\tilde B}_{p} &=&\sum_{i=0}^{p-1}B_i.
 \end{eqnarray*}
We take ${\tilde B}_{0}\in \Zp$ arbitrarily and see that assumption (1) guarantees ${\tilde B}_{p}\equiv 0 \pmod{p}$ for the 1-Lipschitz property.
Given that ${\tilde B}_{p^{n-1}} \in \Zp$ with ${\tilde B}_{p^{n-1}} \equiv 0 \pmod{p^{n-1}}$ ($n\geq2$), from a system of linear equations in
Eqs.(\ref{ls3}) and (\ref{ls4}), we take
 $\{{\tilde B}_{m} \}_{m=p^{n-1}}^{p^n}$ with ${\tilde B}_{p^n} \equiv 0 \pmod{p^{n}}$
 such that  for all $\a =1,\cdots p^{n-1}-1,$
 \begin{eqnarray*}
{\tilde B}_{ip^{n-1} +\a} &=& i{\tilde B}_{p^{n-1}} +\sum_{m=p^{n-1}}^{ip^{n-1} +\a -1} B_m~~~(i=1,\cdots p-1);\\
{\tilde B}_{ip^{n-1}} &=& i{\tilde B}_{p^{n-1}} +\sum_{m=p^{n-1}}^{ip^{n-1}-1} B_m~~~(i=2,\cdots p).
 \end{eqnarray*}
We see that ${\tilde B}_{p^n} \equiv 0 \pmod{p^{n}}$ follows from assumption (2) and check  that
${\tilde B}_{m}~ ( p^{n-1} < m <p^n)$ satisfies the 1-Lipschitz property. This completes the proof.

\end{proof}

The first part of the following result is observed through Lemma 4.41 in \cite{AK}. However, the second part
provides a clue about  coefficient conditions for the ergodicity of 1-Lipschitz functions in terms of the van der Put expansion.

\begin{theorem}\label{axb}
Let $f(x)=\sum_{m=0}^{\infty}B_m \chi(m,x):\Z_p \rightarrow \Z_p$  be a  measure-preserving 1-Lipschitz function of
the form $f(x)= d+\e x + p\Delta g(x)$ for a suitable 1-Lipschitz function $g(x),$ where $\e \equiv 1 \pmod {p}$ and $d \not \equiv 0 \pmod{p}.$
Then (i) the function  $f$ is ergodic.

(ii) We have the following congruence relations:

(1)$B_0 \equiv s \pmod{p}$ for some $ 0 <s<p;$

(2) $\sum_{m=0}^{p-1}B_m  \equiv ps+ \frac{1}{2}(p-1)p\pmod{p^2};$

(3)
\[ \sum_{m=p}^{p^2-1}B_m   \equiv  \frac{1}{2}(p-1)p^3 \equiv \left \{ \begin{array}{ll}
4  \pmod{2^3} & if p=2; \\
 0  \pmod{p^3}    &  if p>2
;
\end{array}
 \right. \]

(4) $B_m \equiv q(m) \pmod{p^{\lfloor {\rm log}_{p}m\rfloor +1}}$ for all $m\geq p;$

(5) $\sum_{m=p^{n-1}}^{p^n -1} B_m  \equiv 0 \pmod{p^{n+1}}$ for all $n\geq 3.$

\end{theorem}

\begin{proof}
It is known that the first assertion follows from Lemma 4.41 \cite{AK}.
For the second assertion,we first note that two simple functions, a constant $d \in \Z_p,$ and $x$ have an explicit expansion in terms of the van der Put series:
 \begin{eqnarray}\label{candx}
d &=& \sum_{m=0}^{p-1} d\chi(m,x); \nonumber \\
x &=& \sum_{m=1}^{p-1} m\chi(m,x)+ \sum_{m\geq p} q(m)\chi(m,x).
 \end{eqnarray}

If we write a 1-Lipschitz function $g(x)=\sum_{m=0}^{\infty}{\tilde B}_m \chi(m,x),$ then
we have from Proposition \ref{lipcond}
\[ B_m  = \left \{ \begin{array}{ll}
   d +\e m +p({\tilde B}_{m+1}-{\tilde B}_{m}) &  ~if~0\leq m \leq p-2;\\
   d+\e(p-1)+ p({\tilde B}_{p}+{\tilde B}_{0} -{\tilde B}_{p-1}) & ~if ~m=p-1;\\
 \e q(m) +p({\tilde B}_{m+1}- {\tilde B}_{m})  & ~if~ m  \not = p^{n-1}-1+m_{n-1}p^{n-1}, 1 \leq m_{n-1} \leq p-1, n \geq 2;\\
\e q(m) +p({\tilde B}_{m+1} -{\tilde B}_{m} -{\tilde B}_{p^{n-1}})  & ~if~ m = p^{n-1}-1+m_{n-1}p^{n-1}, 1 \leq m_{n-1} \leq p-1, n \geq 2.
\end{array}
 \right.
 \]
From these formulas for $B_m,$ it is now straightforward to deduce conditions (1)-(4) together with the assumptions about  $d$ and $\e.$
For condition (5), we have, for all $n \geq 3,$
\begin{eqnarray*}
\sum_{m=p^{n-1}}^{p^n -1} B_m & \equiv& \sum_{m=p^{n-1}}^{p^n -1} B_m -\e q(m)\pmod{p^{n+1}}\\
&=&\sum_{m=p^{n-1}}^{p^n -1}p(\tilde{B}_{m+1}-\tilde{B}_{m}) -p(p-1)\tilde{B}_{p^{n-1}} \\
&=&p\tilde{B}_{p^n}-p^2\tilde{B}_{p^{n-1}} \equiv 0 \pmod {p^{n+1}},
 \end{eqnarray*}
because $\tilde{B}_{m}$ satisfy the 1-Lipschitz property. This completes the proof.
\end{proof}

We provide a partial answer for the converse of Theorem \ref{axb} under some additional condition that is trivially satisfied for the case in which $p$= 2 or 3.
For the first main result,  we provide the sufficient conditions under which a measure-preserving
1-Lipschitz function on $\Z_p$ represented by the van der Put series is ergodic.
The conditions in Theorem \ref{axb} reduce to all conditions in Theorem \ref{putcri2} (iii) for the case  $p=2.$

\begin{theorem}\label{eprop}
Let $f(x)=\sum_{m=0}^{\infty}B_m \chi(m,x):\Z_p \rightarrow \Z_p$  be a  1-Lipschitz function satisfying
all conditions in Theorem \ref{axb} (ii).
If $f$ satisfies the additional condition $B_m \equiv B_0 + m \pmod{p}$ for $ 0 <m<p,$ then $f$ is ergodic.
\end{theorem}

\begin{proof}
By Lemma 4.41 in \cite{AK} or Lemma  \ref{dxperg} in Section 2, it suffices to show that the function $f$ is of
the form $f =B_0 + x + p\Delta g(x)$  with some  1-Lipschitz function $g(x).$
By Theorem \ref{MPinp}, we observe that $f$ is measure-preserving. Indeed, this follows from condition (4) in Theorem \ref{axb} and the additional condition.
We now use the said conditions and Eq.(\ref{candx}) to break  $f(x)$ up as follows:
\begin{eqnarray*}
 f(x) &=&\sum_{m=0}^{p-1}B_m \chi(m,x) +\sum_{m\geq p}( q(m)+ pB_m') \chi(m,x) ~ {\rm with } ~B_m' \equiv 0 \pmod{p^{\lfloor {\rm log}_{p}m\rfloor }} \\
& =& B_0\chi(0,x)+ \sum_{m=1}^{p-1}(B_0 + m)\chi(m,x) + \sum_{m\geq p} q(m)\chi(m,x) + p\sum_{m\geq 0}B_m^{''} \chi(m,x) \\
& =& B_0 + x +p\sum_{m\geq 0}B_m^{''}\chi(m,x)
\end{eqnarray*}

By equating the coefficients of $f$ on both sides of the preceding equation, we have
\[  B_m =
                  \left \{ \begin{array}{ll}
  B_0 + m +pB_m^{''}  & ~if ~0 \leq  m\leq p-1;\\
 q(m) +pB_m^{''}  & ~if~ m \geq p.
 \end{array}
 \right. \]
We use this equation to see that
condition (2) in Theorem \ref{axb} is equivalent to $\sum_{m=0}^{p-1}B_m^{''} \equiv 0 \pmod{p}$ and that conditions (5) and (3) are equivalent to $\sum_{m=p^{n-1}}^{p^n-1}B_m^{''} \equiv 0 \pmod{p^{n}}$
for all $n\geq 2.$
Because $~B_{m}'$ for $m\geq p$ satisfy  the 1-Lipschitz property, so do $B_m^{''}$   for $m\geq p.$ Therefore,
 we see from Proposition \ref{lipcond} that
$\sum_{m\geq 0}B_m^{''} \chi(m,x)=\Delta g(x)$ for some  1-Lipschitz function $g(x),$ and we are done.
\end{proof}

\subsection{Equivalent Statements}
We provide several equivalent conditions that may be needed for a complete description of the ergodicity of  1-Lipschitz functions on
$\Z_p.$
For this, we need to observe the following property for 1-Lipschitz functions.

\begin{lemma}\label{mpid}
Let  $f(x)=\sum_{m=0}^{\infty}B_m \chi(m,x):\Z_p \rightarrow \Z_p$  be a  1-Lipschitz function represented
by the van der Put series.
Then,  for all $n\geq 2,$ we have
$$\sum_{m=p^{n-1}}^{p^n-1}B_m= \sum_{m=0}^{p^n-1}f(m) -p\sum_{m=0}^{p^{n-1}-1}f(m).$$
\end{lemma}

\begin{proof}
For $p^{n-1} \leq m < p^n$ with $n\geq 2,$ write $m=ip^{n-1} +j,$ where $1\leq i <p$ and $0 \leq j <p^{n-1}.$
We use the formula for $B_m$ in Eq.(\ref{Bmform}) to compute $\sum_{m=p^{n-1}}^{p^n-1}f(m)$ as follows:
\[
\begin{split}
\sum_{m=0}^{p^n -1}f(m)-\sum_{m=0}^{p^{n-1}-1}f(m) =\sum_{m=p^{n-1}}^{p^n-1}f(m)&=\sum_{m=p^{n-1}}^{p^n-1}B_m +f(m\underline{~~})\\
 &= \sum_{i=1}^{p-1}\sum_{j=0}^{p^{n-1}-1}B_{ip^{n-1}+j} +\sum_{i=1}^{p-1}\sum_{j=0}^{p^{n-1}-1}f(j)\\
 &= \sum_{m=p^{n-1}}^{p^n-1}B_m+(p-1)\sum_{m=0}^{p^{n-1}-1}f(m).
\end{split}
\]
Then we have the desired result.
\end{proof}

{\bf Remarks} 1. If the 1-Lipschitz function $f=\sum_{m=0}^{\infty}B_m \chi(m,x):\Z_p \rightarrow \Z_p$ satisfies the relationship $f=\Delta g$ for a suitable 1-Lipschitz function $g=
\sum_{m=0}^{\infty}{\tilde B}_m \chi(m,x),$ then it is known from Proposition \ref{lipcond} that for $n\geq 1,$ $$\sum_{m=p^{n-1}}^{p^n-1}B_m  = {\tilde B}_{p^n} -p{\tilde B}_{p^{n-1}}.$$

2. If the additional condition $g(0)=0$ is satisfied, then by Theorem 34.1 in \cite{S}, we have
$$\sum_{m=0}^{p^n-1}f(m)=g(p^n)= {\tilde B}_{p^n}$$
for all $n \geq 1.$

From this point onward, we assume that $f:\Z_p \rightarrow \Z_p$ is a measure-preserving 1-Lipschitz function.
For a nonnegative integer $m,$ we write
\begin{eqnarray}\label{fmn}
f(m)=\sum_{i=0}^{\infty}f_{mi}p^i ~{\rm with}~0 \leq f_{mi} \leq p-1 ~(i=0, 1, \cdots)
\end{eqnarray}
For an integer $n\geq1,$ we define $S_n$ to be
\begin{eqnarray}\label{sn}
S_n=\sum_{m=0}^{p^n -1}f_{mn}.
\end{eqnarray}
From Lemma \ref{mpid}, we immediately see that for all $n\geq2,$

\begin{eqnarray}\label{BnSn}
\sum_{m=p^{n-1}}^{p^n-1}B_m \equiv 0\pmod{p^{n+1}} \Leftrightarrow \sum_{m=0}^{p^n-1}f(m) \equiv p\sum_{m=0}^{p^{n-1}-1}f(m) \pmod{p^{n+1}}.
\end{eqnarray}

Because $f$ is measure-preserving,
the congruence on the right-hand side of Eq.(\ref{BnSn}) is equivalent to rewriting it as
$$ \frac{1}{2}(p^n -1)p^n +S_np^n \equiv p\sum_{m=0}^{p^{n-1}-1}f(m) \pmod{p^{n+1}}.$$
Canceling $p$ out, we have
$$\frac{1}{2}(p^n -1)p^{n-1} +S_np^{n-1} \equiv \sum_{m=0}^{p^{n-1}-1}f(m) \pmod{p^{n}}.$$
Because $f$ is again measure-preserving, we have
$$\frac{1}{2}(p^n -1)p^{n-1}+S_np^{n-1} \equiv \frac{1}{2}(p^{n-1} -1)p^{n-1} +S_{n-1}p^{n-1} \pmod{p^{n}}.$$
Canceling $p^{n-1}$ out gives
$$\frac{1}{2}(p-1)p^{n-1}+S_n \equiv S_{n-1} \pmod{p}~~(n\geq2).$$
This gives the following congruence:
\[S_n \equiv  \left \{ \begin{array}{ll}
  S_{n-1} \pmod{p}~~(n\geq2)~if~ p \not =2;\\
  S_{n-1} \pmod{2}~~(n\geq3)~ if~ p=2.
\end{array}
 \right.
 \]
On the other hand, because $f$ is measure-preserving, by proposition \ref{Mp-sumform}, we obtain
\begin{eqnarray}\label{sumbm}
\sum_{m=p^{n-1}}^{p^n-1}B_m \equiv \frac{1}{2}(p-1)p^{2n-1} +T_np^n \pmod{p^{n+1}}.
\end{eqnarray}
This gives the following equivalence: For the case $(p,n)$ in which $n\geq2$ if the prime $p$ is odd, and $n \geq 3$ otherwise,
 we have either
$$\sum_{m=p^{n-1}}^{p^n-1}B_m \equiv 0 \pmod{p^{n+1}} \Leftrightarrow T_n \equiv 0\pmod{p} ,$$
or
$$\sum_{m=p^{n-1}}^{p^n-1}B_m \equiv T_np^n \not \equiv 0 \pmod{p^{n+1}} \Leftrightarrow T_n \not \equiv 0\pmod{p} .$$

For the case $(p,n)=(2,2),$ we have from Eq.(\ref{sumbm}) that  either
$$\sum_{m=2}^{3}B_m \equiv 0 \pmod{2^{3}} \Leftrightarrow T_2 \equiv 1\pmod{2},$$
or
$$\sum_{m=2}^{3}B_m \equiv 4 \pmod{2^{3}} \Leftrightarrow T_2 \equiv 0 \pmod{2}.$$
From Lemma \ref{mpid} and Eq.(\ref{sumbm}) we deduce the following congruence: For all $n\geq2,$ we have
$$ T_n \equiv S_n -S_{n-1} \pmod{p}.$$

In sum, we have the following equivalence:
\begin{theorem}\label{mpeqp2}
Let $f(x)=\sum_{m=0}^{\infty}B_m \chi(m,x):\Z_p \rightarrow \Z_p$ be a measure-preserving 1-Lipschitz function represented
by the van der Put series. In addition, let $b_n, T_n$ and $S_n$ be defined as in Proposition \ref{Mp-sumform} and in Eq.(\ref{sn}).
Then we have the following equivalence:

(1)$n=2:$ \\
(a) $p=2:$
$$\sum_{m=2}^{2^2-1}B_m \equiv4 \pmod{2^{3}}   \Leftrightarrow \sum_{m=2}^{2^2-1}b_m  \equiv 2 \pmod{2^{3}} $$
$$
\Leftrightarrow S_2  \equiv S_1  \pmod{2}
\Leftrightarrow T_2\equiv 0 \pmod{2};$$
 or
\noindent $$\sum_{m=2}^{2^2-1}B_m \equiv0 \pmod{2^{3}}   \Leftrightarrow \sum_{m=2}^{2^2-1}b_m  \equiv 0 \pmod{2^{3}}$$
$$
\Leftrightarrow S_2  \equiv S_1 +1 \pmod{2}
\Leftrightarrow T_2\equiv 1\pmod{2}.$$\\
(b) $p>2:$
$$\sum_{m=p}^{p^2-1}B_m \equiv rp^2 \pmod{p^{3}}  \Leftrightarrow \sum_{m=p}^{p^2-1}b_m \equiv rp \pmod{p^2} $$
$$
\Leftrightarrow S_{2} \equiv S_{1}
 +r \pmod{p}
\Leftrightarrow T_2\equiv r \pmod{p}.$$

(2) $n\geq3$ and any prime $p:$\\
$$\sum_{m=p^{n-1}}^{p^n-1}B_m \equiv rp^n \pmod{p^{n+1}}  \Leftrightarrow \sum_{m=p^{n-1}}^{p^n-1}b_m \equiv rp \pmod{p^2} $$
$$
\Leftrightarrow S_{n} \equiv S_{n-1}
 +r \pmod{p}
\Leftrightarrow T_n\equiv r \pmod{p}.$$

\end{theorem}



\subsection{Alternative proofs of Anashin's and Anashin et al 's results }

In this section, we use Theorem \ref{mpeqp2} to provide an alternative proof of Theorem \ref{putcri2} (iv).
For this, we need the following lemma, which is an analog in $\Z_2$ of the result for the formal power series ring $\F_2[[T]]$ over the
field $\F_2$ of two elements (see Lemma 1 in \cite{LSY}).
\begin{lemma}\label{p=2equi}
Let $f:\Z_2 \rightarrow \Z_2$ be a measure-preserving 1-Lipschitz function such that $f$ is transitive modulo $2^n, n \geq 1.$
Then $f$ is transitive modulo $2^{n+1}$ if and only if $S_n$ is odd, where $S_n$ is defined as in Eq.(\ref{sn}).
\end{lemma}

\begin{proof}
($\Rightarrow$): We put $R_{<n}=\{0, 1, \cdots, 2^{n}-1  \}$ for a complete set of the least nonnegative representatives of $\Z_2/2^n\Z_2.$
When we consider the trajectory of $f(x)$ modulo $2^k,$ we view $x$ and $f(x)$ as elements whose representatives are in $R_{<k}.$
If $f$ is transitive modulo $2^{n+1},$ then there exist $x_0, x_1 \in R_{<n}$ such that $f(x_0)=x_1 + 2^n.$
Starting with $x_0$ as the initial point, we list the trajectory of $f$ modulo $2^{n+1}$ as follows:

\begin{alignat}{3}
 x_0 &\rightarrow  f(x_0) \cdots &\rightarrow   \cdots &\rightarrow f^{2^n-1}(x_0) &\rightarrow \nonumber \\
\rightarrow f^{2^n}(x_0)  &\rightarrow f^{2^n +1}(x_0)\cdots  &\rightarrow \cdots &\rightarrow f^{2^{n+1}-1}(x_0) &\rightarrow \label{tarrary}  \\
\rightarrow  f^{2^{n+1}}(x_0)&=x_0 +2^{n+1}u \pmod{2^{n+1}} \nonumber
\end{alignat}
where $u \in 1+2\Z_2.$
Because $f$ is both measure-preserving and transitive modulo $2^{n+1},$ we have $f^{2^n}(x_0)=x_0 +2^n.$ We use this relationship to iteratively derive
the relationship
\begin{eqnarray}\label{fpow}
f^{2^n +i}(x_0) \equiv f^{i}(x_0) +2^n \pmod{2^{n+1}}~~( 0 \leq i \leq 2^n-1).
\end{eqnarray}
We claim that $S_n$ is odd and thus that $S_n= \# \{ 0 \leq m \leq 2^n-1:  f_{mn} =1 \}$ is odd, where $f_{mn}$ is defined in Eq.(\ref{fmn}).
If there exists a number in $R_{<n}$ other than $x_0$ mapped by $f$ to an element in $R_{<n} +2^n$ in the first row of the diagram in Eq.(\ref{tarrary}),
then there exists another element in $R_{<n} +2^n$ that maps to an element in $R_{<n}.$ By the relationship in Eq. (\ref{fpow}), we see that
there must be an element in $R_{<n}$ that is mapped by $f$ to an element in $R_{<n} +2^n$ in the second row.
This implies that the total number of elements  in $R_{<n}$ that are mapped by $f$ to an element in $R_{<n} +2^n$ is odd and thus that $S_n$ is odd.

Conversely, assuming that $S_n$ is odd, we see that  there exist $x_0, x_1 \in R_{<n}$ such that $f(x_0)=x_1 + 2^n.$
From the above diagram, because $f$ is transitive modulo $2^n,$ we observe that
the elements of the first row as well as those in the second row are distinct modulo $2^n.$
We now show that $f^{2^n}(x_0)=x_0 +2^n.$ Otherwise, we have $f^{2^n}(x_0)=x_0,$ and  therefore we see that $\# \{ 0 \leq m \leq 2^n-1:  f_{mn} =1 \}$ is even,
which is a contradiction. As in the "only if" part, we use $f^{2^n}(x_0)=x_0 +2^n$ to derive the relationship in Eq.(\ref{fpow}). Therefore, these relationships imply
that  the trajectory of $f$ modulo $2^{n+1}$ are all distinct modulo $2^{n+1}.$ Hence, $f$ is transitive modulo $2^{n+1}.$

\end{proof}

\begin{theorem}\label{ayresult}
Let $f(x)=\sum_{m=0}^{\infty}B_m\chi(m,x) :\Z_2 \rightarrow \Z_2$ be a 1-Lipschitz function.
Then $f$ is ergodic  if and only if all conditions in Theorem \ref{putcri2} (iii) are satisfied.

\end{theorem}
\begin{proof}
We see that the "if" part follows immediately from Theorem \ref{eprop} because the additional condition there is trivially satisfied for $p=2.$
For the "only if " part, we note that $S_1=1,$ so this direction follows from Lemma \ref{p=2equi} and Theorem \ref{mpeqp2}.
\end{proof}

As a corollary, we reproduce Corollary \ref{1+xDE}.
\begin{corollary}\label{1+x}
Let $f:\Z_2 \rightarrow \Z_2$ be a 1-Lipschitz function.
Then, (1) $f$ is measure-preserving if and only if $f$ is of the form $f(x)=d+x +2g(x)$ for some
 2-adic integer $d\in \Z_2$ and some 1-Lipschitz function $g(x).$

(2) $f$ is ergodic if and only if $f$ is of the form $f(x)= 1+x +2\Delta g(x)$ for some 1-Lipschitz function $g(x).$
\end{corollary}

\begin{proof} For the first assertion, the "if" part  follows from Proposition \ref{mper} (3). And the "only if " part
comes from Theorem \ref{MPinp}, because the conditions there is necessary in the case $p=2.$

For the second assertion, the "if" part  follows from Lemma 4.41 in \cite{AK} and the "only if " part follows from Theorems \ref{ayresult} and \ref{eprop}.
\end{proof}

We now use  Corollary \ref{1+x} to provide an alternate proof of Theorem \ref{Ana-binom} (iv).
For this, we first need to provide the 1-Lipschitz conditions in Theorem \ref{Ana-binom} (i). However, we just mention that this property can be proved in the similar way by using the well-known binomial formula in \cite{LSY, Ya} for Carlitz polynomials over functions fields.
\begin{corollary}
Let $ f(x)=\sum_{m= 0}^{\infty}a_m\binom{x}{m}:\Z_2 \rightarrow \Z_2$ be a 1-Lipschitz function.
Then $f$ is ergodic if and only if all conditions in Theorem \ref{Ana-binom} (iv) are satisfied.
\end{corollary}
\begin{proof}
It follows from Corollary \ref{1+x}.
\end{proof}%

\section{ An Application}

In this final section, we use Theorem \ref{putcri2} to derive a characterization for the ergodicity of a polynomial over $\Z_2$ in term of its
coefficients. For simplicity, we take a polynomial $f \in \Z_2[x]$ with $f(0)=1.$ That is, let $f =a_dx^d +a_{d-1}x^{d-1}+ \cdots + a_1x+1 $ be a polynomial of degree $d$ over $\Z_2.$ Then we set
$$ A_0 =\sum_{i\equiv 0 \pmod{2}, i >0} a_i, ~~ A_1= \sum_{i\equiv 1 \pmod{2}} a_i.$$

\begin{theorem}\label{ploy}
The polynomial $f$ is ergodic over $\Z_2$ if and only if the following conditions are simultaneously satisfied:
\begin{eqnarray*}
 a_1 \equiv 1 \pmod{2};\\
A_1 \equiv 1 \pmod{2};\\
A_0 +A_1 \equiv 1 \pmod{4};\\
a_1 + 2a_2 +A_1 \equiv 2 \pmod{4}.
\end{eqnarray*}
\end{theorem}
\begin{proof}
From Theorem \ref{putcri2} (iii) or Theorem \ref{ayresult}, we derive the equivalent conditions as required.
Because
$B_0+ B_1 =f(0)+f(1)=2+A_0 +A_1,$ we can easily see that
$B_0 +B_1 \equiv 3 \pmod{4}$ is equivalent to $A_0 +A_1 \equiv 1 \pmod{4}.$
Because  $B_2=2b_2=f(2)-f(0), B_3=2b_3=f(3)-f(1),$  we see that
$b_2+b_3 \equiv 2 \pmod{4}$ is equivalent to $f(2)-f(0)+f(3)-f(1) \equiv 4 \pmod{8}.$
Because  $f(3)=\sum_{i=0}^{d}a_i3^i \equiv 1 +A_0 +3A_1 \pmod{8},$ we have the following equivalence:
$$b_2+b_3 \equiv 2 \pmod{4} \Longleftrightarrow a_1 + 2a_2 +A_1 \equiv 2 \pmod{4}.$$
For all $m\geq 2,$ we have
\begin{eqnarray}
B_m =f(m)-f(m\underline{~~}) &=&\sum_{i=1}^{d}a_i (m^i- m\underline{~~}^i ) \nonumber \\
 &=&\sum_{i=1}^{d}a_i\sum_{j=1}^{i}\binom{i}{j}m\underline{~~}^{i-j}q(m)^j \nonumber \\
 &=&\sum_{j=1}^{d} \left(\sum_{i=j}^{d}\binom{i}{j}a_im\underline{~~}^{i-j} \right) q(m)^j.\label{pmod4}
\end{eqnarray}
The preceding equation implies that condition (4) is equivalent to $f'(m\underline{~~}) \in 1+2\Z_2$ for all $m \geq2,$ where $f'(x)$ is the derivative of $f.$
Equivalently, $f'(0)=a_1 \in 1+2\Z_2$ and $f'(1) \equiv A_1 \in 1+ 2\Z_2.$
From Eq.(\ref{pmod4}), we can deduce that for all $m\geq 2,$
$b_m \equiv f'(m\underline{~~}) \pmod{q(m)}.$ From this, we obtain the following congruence: For $n\geq 3,$
\[
\begin{split}
\sum_{m=2^{n-1}}^{2^n -1}b_m \equiv \sum_{m=2^{n-1}}^{2^n -1}f'(m\underline{~~}) &=\sum_{j=0}^{2^{n-1} -1}f'(j) \\
 & \equiv 2^{n-3}(f'(0)+f'(1)+f'(2)+f'(3)) \\
 & \equiv 2^{n-2}(A_1 -a_1) \pmod{4}.
\end{split}
\]
This congruence implies that condition (5) is redundant.  This completes the proof.
\end{proof}

\noindent{\bf Remarks}
1. We first mention that all conditions in  Theorem \ref{ploy} are easily proved to be equivalent to those in \cite{DP} or \cite{La}.

2. We point out that the result for this theorem extends to a class of analytic functions on $\Z_p$ by which we mean those functions
represented by the Taylor series on all $\Z_p.$

3. The characterization for the  ergodicity of  1-Lipschitz functions provides a clue for a complete description of the necessary and sufficient
conditions for a polynomial function on $\Z_p$ in terms of its coefficients, as in Theorem \ref{ploy}.

4. Future research should use the results in this paper, particularly those for Theorems \ref{eprop} and \ref{mpeqp2}, to provide a complete description of an ergodic 1-Lipschitz function $ \Z_p$ represented by the van der Put series for all odd primes  $p.$


\end{document}